\documentclass[11pt]{article}
\usepackage{graphicx}
\usepackage{bbold,amsthm,amssymb,amsmath,pgfplots,breqn,blindtext,titlesec,hyperref,setspace,mathtools}
\usepackage[margin=2.8cm]{geometry}
\pgfplotsset{width=10cm,compat=1.9}
\newtheorem{theorem}{Theorem}[section]
\newtheorem{proposition}{Proposition}[section]
\newtheorem{corollary}{Corollary}[section]
\newtheorem{lemma}{Lemma}[section]

\theoremstyle{definition}
\newtheorem{remark}{Remark}[section]
\newtheorem{definition}{Definition}
\newtheorem{example}{Example}

\renewcommand{\L}{\mathcal{L}}
\newcommand{\A}{\mathcal{A}}
\newcommand{\C}{\mathbb{C}}
\newcommand{\R}{\mathbb{R}}
\newcommand{\N}{\mathbb{N}}
\newcommand{\M}{\mathcal{M}}

\DeclareMathOperator{\supp}{supp}

\begin{document}

\begin{titlepage}

\newcommand{\HRule}{\rule{\linewidth}{0.5mm}} 

\center 
\includegraphics[width=10cm]{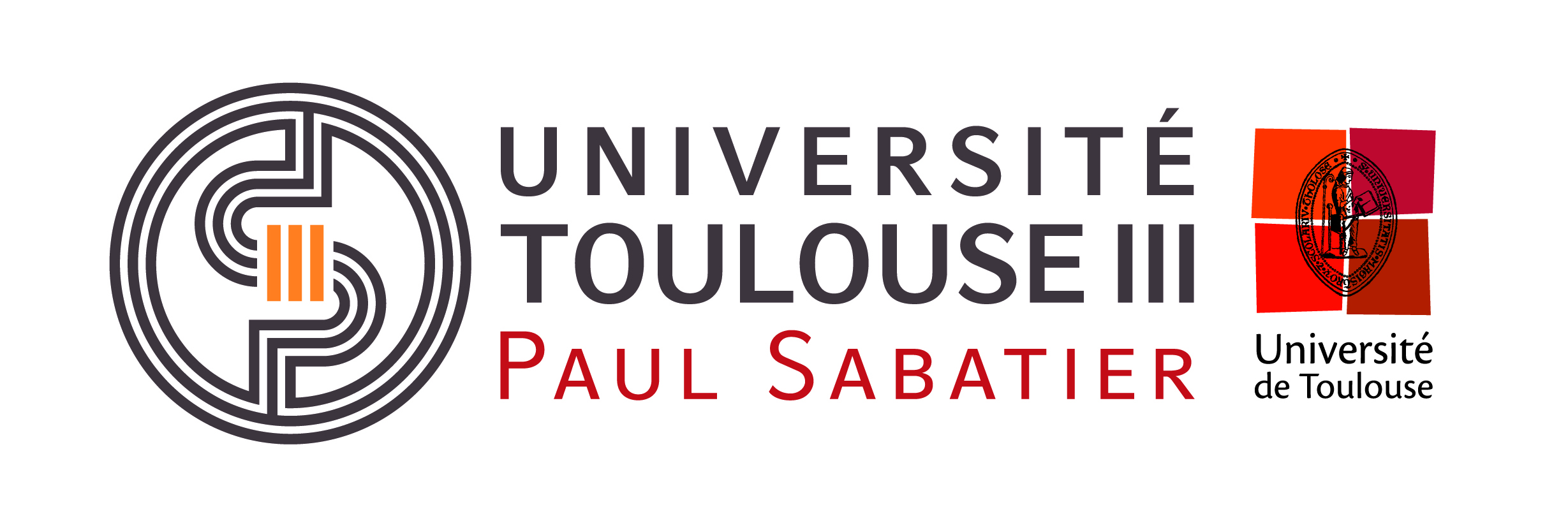}\\[2.5cm]

\textbf{\LARGE Master Thesis}\\[1.5cm]  

\HRule \\[0.4cm]
{ \huge \bfseries Approximate and null controllability of a parabolic system with coupling terms of order one}\\[0.4cm]
\HRule \\[3.5cm]

\begin{minipage}{0.4\textwidth}
\begin{flushleft} \large
\emph{Supervisor}\\

Franck \textsc{Boyer}\\ 
\end{flushleft}
\begin{flushleft} \large
\emph{Student}\\
Amélie \textsc{Dupouy}\\ 

\end{flushleft}

\end{minipage}\\[3cm]

{\large March - June 2023}\\[2cm] 

\vfill 

\end{titlepage}
\clearpage
\shipout\null
\tableofcontents 

\newpage
\section{Introduction}

\subsection{Presentation of the problem}
We consider in this work the following type of systems 
\begin{equation} \label{eq:intro}
    \left \{
    \begin{aligned}
        \partial_t y + \mathcal{A} y &= \mathcal{B}v \\
        y(0) = y^0
    \end{aligned}
    \right.
\end{equation}
where $y$ is the unknown, $\mathcal{A}$ is an elliptic operator, $v$ is the control function, and $\mathcal{B}$ is the control operator. \\
Our goal is to study two controllability properties of the system, called \textit{approximate controllability} and \textit{null controllability}. These notions will be defined precisely in further sections; for now, we give a general idea. \\
A system is said to be approximately controllable at time $T$ if its solution can be driven to be as close as wanted to a target data at the instant $T$. This is the subject of the second section of this report. \\
The notion of null controllability is more precise. Indeed, our target data is now precisely $0$, and not only do we want to steer the solution close to it, but precisely on it. We will study this issue in the third and last part of this report. 
\subsection{Two-component cascade case}
We consider the following problem:
\begin{equation} \label{eq:D}
\left \{
\begin{aligned}
    \partial_{t} y + \L y &= \begin{pmatrix} 0 & 0 \\ -q(x) & 0 \end{pmatrix}y + \begin{pmatrix} 0 & 0 \\ -p(x) & 0 \end{pmatrix} \partial_{x}y + \mathbb{1}_{\omega} Bv \\
    y(0,\cdot) = y(\pi,\cdot) &= 0 \\
    y(\cdot, 0) &= y^{0}
\end{aligned}
\right.
\end{equation}

where $T>0$, $y \in C^{0}([0,T], L^{2}(0, \pi)^2)$ is the unknown, $y^{0} \in L^{2}(0, \pi)^2$ is the initial data, $B = \begin{pmatrix}
        1 \\ 0
\end{pmatrix}$, $v \in L^{2}((0,T) \times (0, \pi))$ is the control acting only in an open subset $\omega$ of $(0, \pi)$, $p \in W^{1, \infty}(0, \pi)$, $q \in L^{\infty}(0, \pi)$ are the coupling terms, and $\L = \mathfrak{L} Id$, with $\mathfrak{L} $ being the scalar operator defined by 
$$ \mathfrak{L}   = -\partial_{x}(\gamma(x)\partial_{x}\cdot) + \gamma_{0}(x) \cdot$$
with domain $\mathcal{D}(\mathfrak{L}) = \{u \in H_{0}^{1}(0,\pi), \: \mathfrak{L}  u \in L^{2}(0,\pi)\}$, $\gamma,\gamma_0 \in L^{\infty}(0,\pi)$ and $\displaystyle \inf_{(0,\pi)} \gamma > 0$. 
\clearpage

The study conducted during this internship is mainly based on two papers. The first one, by F. Boyer and G. Olive \cite{boyer}, states results about the approximate controllability of parabolic systems with coupling terms of order zero. The second by M. Duprez \cite{duprez} deals with the case with coupling terms of order one, in the particular case where the control domain is an interval and the parabolic operator is the Laplacian operator. 

The goal of this internship is then to generalise the results given in these articles, to a case where there is a coupling term of order one and the control domain $\omega$ is not necessarily an interval. To study the approximate controllability, we mainly use two results: the Fattorini-Hautus test (see Theorem \ref{th:fh}) and a unique continuation property (Theorem \ref{th:uc}). For the null controllabilty, we will use the moments method in order to find a suitable control function for our system. 
\vspace{1em}

The first part of the study will consist in establishing useful spectral and technical results for the rest of the paper. \\
The second section will focus on a result stated by M. Duprez in \cite{duprez}, given in Theorem \ref{th:1}. It appears that even if this theorem is correct in most cases, it can happen that the hypotheses are verified and yet the result does not apply. Hence, we propose a new version of this theorem for the approximate controllability (see Theorem \ref{th:AC}), and give a counter-example to the original result in Section \ref{counterex}. \\
The last part of this report deals with the null controllabilty of our system, in the most possible general case with $\omega$ being an open subset of our domain, and $\mathcal{L}$ the general parabolic operator defined earlier. 
\vspace{1em}

Now, let us define the two notions of controllability that will be studied:
\begin{itemize}
    \item System \eqref{eq:D} is said to be \emph{approximately controllable} at time $T$ if for all $\varepsilon>0$ and all $y^{0}, y^{1} \in L^{2}(0,\pi)^{2}$ there exists a control $v \in L^{2}(Q_{T})$ such that the solution to the system satisfies
    $$
    \Vert y(T)-y^{1} \Vert_{L^{2}(0, \pi)^{2}} \leq \varepsilon.
    $$
    \item System \eqref{eq:D} is said to be \emph{null controllable} at time $T$ if for every initial condition $y^{0} \in L^{2}(0,\pi)^{2}$ there exists a control $v \in L^{2}(Q_{T})$ such that the solution to the system satisfies
    $$
    y(T) \equiv 0 \quad \textnormal{in} \: (0, \pi).
    $$
\end{itemize}
We know that null controllability at some time $T$ implies approximate controllability at this same time $T$.
In \cite{duprez}, the following result is stated. 
\clearpage
\begin{theorem} \label{th:1}
    Let us suppose that $p \in W^{1,\infty}(0,\pi) \: \cap \: W^{2,\infty}(\omega)$, $q \in L^{\infty}(0,\pi) \: \cap \:  W^{1,\infty}(\omega)$ and
    $$
    (\supp p \cup \supp q) \cap \omega \neq \emptyset
    $$
    Then the system is null controllable at any time T.
\end{theorem}
We will see in further parts that this theorem cannot be true,
for we can exhibit cases in which the hypotheses are satisfied,
and yet the system is not approximately controllable, hence not null controllable. 
\vspace{1em}

In order to study this system, we will need its adjoint system, that is:
\begin{equation} \label{eq:dual}
\left \{
    \begin{aligned}
         -\partial_{t} \xi + \mathcal{A}^* \xi &= 0  &\textnormal{in} \: (0,T) \times (0,\pi)  \\
         \xi(T) &= \xi_{F} &\textnormal{in} \: (0, \pi)
    \end{aligned}
\right.
\end{equation}

where we define $\mathcal{A}^* = \L  - A_{0}^{*}(-q(x)\cdot + \partial_{x}(p(x) \cdot))$ with $A_{0} = \begin{pmatrix} 0 & 0 \\ 1 & 0 \end{pmatrix}$. \\
The solution $\xi$ to this system is given by 
\begin{equation}
    \xi(t) = e^{-(T-t)\A^*} \xi_F
\end{equation}
where $\left( e^{-\cdot \A^*} \right)$ is the semigroup generated in $L^2(0,\pi)^2$ by $\A^*$.

\section{Spectral properties of the adjoint operator and main tools}
In order to study the controllability of our system, we will need to use some spectral properties of the operator $\A^*$.

\subsection{Spectrum, eigenfunctions and biorthogonal basis}

The operator $\mathfrak{L}$ is selfadjoint and admits a countable set of eigenvalues $\Lambda$ that are all geometrically simple. For each $\lambda \in \Lambda$ we denote $\phi_{\lambda}$ a normalized associated eigenfunction. The set $\left(\phi_{\lambda}\right)_{\lambda \in \Lambda}$ of such eigenfunctions forms a Hilbert basis of $L^2(0,\pi)$.\\
Let us define the following quantity, for all $\lambda \in \Lambda$ :

\begin{center}
$I_{\lambda}(p,q)\coloneqq \displaystyle \int_{0}^{\pi}\left(q-\frac{1}{2} \partial_{x}p\right)\phi_{\lambda}^{2}$.
\end{center}

\begin{lemma}
    The following problem has solutions in $L^2(0,\pi)$: 
    \begin{equation} \label{pbdefpsi}
    \left \{
    \begin{aligned}
        \mathfrak{L} \psi - \lambda \psi &= \partial_{x}(p\phi_{\lambda}) - q\phi_{\lambda} + I_{\lambda}(p,q) \phi_{\lambda} \\
        \psi(0) &= \psi(\pi) = 0.
    \end{aligned}
    \right.
    \end{equation}
\end{lemma}

\begin{proof}
    We have
    $$
    \begin{aligned}
        \displaystyle \int_{\Omega} (\partial_{x}(p\phi_{\lambda}) - q\phi_{\lambda} + I_{\lambda}(p,q) \phi_{\lambda})\phi_{\lambda} &= - I_{\lambda}(p,q) + I_{\lambda}(p,q)  =0.
    \end{aligned}
    $$
    Thus $\partial_{x}(p\phi_{\lambda}) - q\phi_{\lambda} + I_{\lambda}(p,q) \phi_{\lambda} \in \textnormal{ span}(\phi_{\lambda})^{\bot} = \ker{(\mathfrak{L} - \lambda)}^\bot$. Since $\mathfrak{L}$ is selfadjoint, this means
    $$\partial_{x}(p\phi_{\lambda}) - q\phi_{\lambda} + I_{\lambda}(p,q) \phi_{\lambda} \in \textnormal{ im} (\mathfrak{L} - \lambda).$$
\end{proof}

\begin{definition}
    We define $\psi_{\lambda}$ to be the unique solution of \eqref{pbdefpsi} that satisfies
    \begin{equation} \label{eq:ortho}
        (\phi_{\lambda}, \psi_{\lambda})_{L^2(\omega)} = 0.
    \end{equation}
\end{definition}

\begin{proposition} \label{prop:spec}
    \begin{enumerate}
        \item The spectrum of $\A^*$ is exactly $\Lambda$ and is only composed of eigenvalues.
        \item The eigenvalue $\lambda$ is geometrically simple if and only if $I_{\lambda}(p,q) \neq 0$. In that case, $\Phi_{\lambda}\coloneqq 
            \begin{pmatrix}
                \phi_{\lambda} \\
                0
            \end{pmatrix}$ 
        is an eigenfunction, and $\Psi_{\lambda}\coloneqq
            \begin{pmatrix}
                \psi_{\lambda} \\
                \phi_{\lambda}
            \end{pmatrix}$
        an associated generalized eigenfunction. More precisely, we have:
        $$
        \left \{
        \begin{aligned}
            (\A^* - \lambda)\Phi_{\lambda} &= 0 \\
            (\A^* - \lambda)\Psi_{\lambda} &= I_{\lambda}(p,q) \Phi_{\lambda}.
        \end{aligned}
        \right.
        $$
        \item The eigenvalue $\lambda$ is geometrically double if and only if $I_{\lambda}(p,q) = 0$. In that case, $\Phi_{\lambda}\coloneqq 
            \begin{pmatrix}
                \phi_{\lambda} \\
                0
            \end{pmatrix}$  and
        $\Psi_{\lambda}\coloneqq
            \begin{pmatrix}
                \psi_{\lambda} \\
                \phi_{\lambda}
            \end{pmatrix}$
        are eigenfunctions of the operator $\A^*$ associated to the eigenvalue $\lambda$. More precisely, we have:
        $$
        \left \{
        \begin{aligned}
            (\A^* - \lambda)\Phi_{\lambda} &= 0 \\
            (\A^* - \lambda)\Psi_{\lambda} &= 0. 
        \end{aligned}
        \right.
        $$
    \end{enumerate}
\end{proposition}

\begin{definition}
    We denote $\Lambda_1$ the set of all the simple eigenvalues of $\A^*$, and $\Lambda_2$ the set of all the double ones.
\end{definition}

\begin{proof}
    We begin with the proof of \textit{1.}. \\
    First, since $\A^*$ has compact resolvent, its spectrum is only composed of its eigenvalues. \\
    Let $s \in \C$, $u \in L^2(0,\pi)^2$ such that $\A^* u = su$. We have 
    $$\A^* u =su \Leftrightarrow \left \{ \begin{aligned}
        \mathfrak{L} u_1 + qu_2 - \partial_{x}(pu_2) &= su_1 \\
        \mathfrak{L} u_2 \phantom{\; \: + qu_2 - \partial_{x}(pu_2)} &= su_2
    \end{aligned}
    \right.$$
If $u_2=0$, the first equation gives that $s$ is necessarily an eigenvalue $\lambda$ of $\mathfrak{L}$, and if $u_2 \neq 0$, the second equation gives the same result. This means $\sigma(\A^*) \subseteq \Lambda$.
    Conversely, take $\lambda \in \Lambda$. We want to show that there exists a vector $u \in \mathcal{D}(\A^*)$ such that $(\A^* - \lambda)u = 0$. Let $u \in \mathcal{D}(\A^*)$.
    \begin{equation}
        \begin{aligned}
            (\A^*-\lambda)u = \begin{pmatrix}
                \mathfrak{L} u_1 + qu_2 - \partial_{x}(pu_2) - \lambda u_1 \\
                \mathfrak{L} u_2 - \lambda u_2.
            \end{pmatrix}
        \end{aligned}
    \end{equation}
    Let us choose $u_2 = 0$, and $u_1 = \phi_{\lambda}$. Then both of the coordinates are zero, and we can conclude that $\lambda$ is an eigenvalue of $\A^*$. The first point is then proved. 
\vspace{1em}

Now, we prove the second and third points. In both cases, we know by definition that $\Phi_{\lambda}$ is an eigenfunction of $\A^*$ associated to the eigenvalue $\lambda$. We begin with the third point, and assume that $I_{\lambda} = 0$. Then $\psi_{\lambda}$ is the solution to 
    $$
    \left \{
    \begin{aligned}
        \mathfrak{L} \psi - \lambda \psi - \partial_{x}(p\phi_{\lambda}) + q\phi_{\lambda} &= 0 \\
        \psi(0) = \psi(\pi) &= 0.
    \end{aligned}
    \right.
    $$
    Hence $\Psi_{\lambda}$ is an eigenfunction of $\A^*$ associated with $\lambda$, and it is linearly independant from $\Phi_{\lambda}$, so the eigenvalue $\lambda$ is double, and we proved \textit{3.}. \\
    We use the same reasoning to prove the second point, $I_{\lambda}(p,q) \neq 0$ implying that $\Psi_{\lambda}$ is now a generalized eigenfunction of $\A^*$ associated with $\lambda$.
\end{proof}

\subsection{Useful tools}

To study the controllabilty of our system we will mainly use two tools that we introduce in this section. The first one is a very practical characterization of the approximate controllabilty, called the Fattorini-Hautus test, given in \cite{boyer}.
\begin{theorem} \label{th:fh}
    System \eqref{eq:D} is approximately controllable if and only if for any $s \in \C$ and any $u \in \mathcal{D}(\A^*)$ we have
    \begin{equation} \label{eq:fh}
    \left.
        \begin{aligned}
             (\A^*-s) u &=0 &\textnormal{in} \: (0, \pi)  \\
             B^{*} u &=0 &\textnormal{in} \: \omega
        \end{aligned}
        \right \} \Rightarrow u = 0.
    \end{equation}
\end{theorem}
Then, we also need a unique continuation property in order to determine whether a non-homogeneous scalar problem admits a solution that vanishes identically in a subset of the whole domain. \\
First, let us define a family of vectors that will be very useful in the rest of this work.

\begin{definition}
We denote by $\mathcal{C} \left( \overline{(0,\pi) \setminus \omega}\right)$ the set of all connected components of  $\overline{(0,\pi) \setminus \omega}$. For all $C \in \mathcal{C} \left(\overline{(0,\pi) \setminus \omega}\right)$ and $f \in L^{1}(0,\pi)$ we define $M_{\lambda}(f, C) \in \R^{2}$ as follows
$$
M_{\lambda}(f, C) = \left \{ 
\begin{aligned}
    \begin{pmatrix}
        \displaystyle \int_{C} f \phi_{\lambda} \\
        0
    \end{pmatrix}
    & \: \textnormal{if} \: C \cap \{ 0, \pi \} \neq \emptyset, \\
    \begin{pmatrix}
        \displaystyle \int_{C} f \phi_{\lambda} \\ 
        \displaystyle \int_{C} f \Tilde{\phi}_{\lambda}
    \end{pmatrix}
    & \: \textnormal{if} \: C \cap \{ 0, \pi \} = \emptyset.
\end{aligned}
\right.
$$
Then we define the following family of vectors of $\R^{2}$: 
$$
\M_{\lambda}(f,\omega) = (M_{\lambda}(f, C))_{C \in \mathcal{C} \left( \overline{(0,\pi) \setminus \omega}\right)} \in (\R^{2})^{\mathcal{C} \left( \overline{(0,\pi) \setminus \omega}\right)}.
$$
\end{definition}

We can now state and prove the following unique continuation property, also stated in \cite{boyer}.

\begin{theorem} \label{th:uc}
    Let $F \in L^{2}(0, \pi)$ and $\omega$ be a non-empty open subset of $(0, \pi)$.  Let $\lambda \in \Lambda$. There exists a solution $u \in \mathcal{D}(\mathfrak{L})$ to the following problem
    \begin{equation} \label{eq:uc1}
        \left \{
        \begin{aligned}
            \mathfrak{L}  u - \lambda u &= F  &\textnormal{in} \: (0,\pi)  \\
             u &= 0 &\textnormal{in} \: \omega 
        \end{aligned}
        \right.
    \end{equation}
    if and only if
    \begin{equation} \label{eq:uc2}
        \left \{
        \begin{aligned}
            F &= 0 \quad \textnormal{in} \: \omega \\
            \M_{\lambda}(F,\omega) &= 0.
        \end{aligned}
        \right.
    \end{equation}
\end{theorem}
\begin{proof}
    We first make a preliminary computation. Let $(\alpha, \beta) \subseteq (0,\pi)$, $u \in \mathcal{D}(\mathfrak{L})$ a solution of 
    $$
    \mathfrak{L}u - \lambda u = F.
    $$
    Then, let $v \in L^2(0,\pi)$ such that $\mathfrak{L} v - \lambda v = 0$; obviously:
    \begin{equation} \label{eq:ipp}
        \displaystyle \int_{\alpha}^{\beta} Fv = - \left[ (\gamma u')(\beta)v(\beta) - u(\beta)(\gamma u')(\beta) \right] + \left[ (\gamma u') (\alpha)v(\alpha) - u(\alpha)(\gamma v')(\alpha)\right]
    \end{equation}
($\Rightarrow$) Suppose that there exists $u$ such that \eqref{eq:uc1} is satisfied. \\
Since $u=0$ in $\omega$, it is clear from the equation that $F=0$ in $\omega$, and by an argument of continuity, we also have that $u = \gamma u' = 0$ in $\omega$. \\
Now consider $C=(\alpha, \beta)$ a connected component of $\overline{(0,\pi) \setminus \omega}$. Necessarily, $\alpha \in \{0, \pi\}$ or $\alpha \in \overline{\omega}$ (and the same goes for $\beta$). We observe that
$$
\left \{
\begin{aligned}
    \alpha \in \{0,\pi\} &\Rightarrow u(\alpha) = 0 \textnormal{ and } \phi_{\lambda}(\alpha)=0 \\
    \alpha \in \overline{\omega} &\Rightarrow u(\alpha) = 0 \textnormal{ and } (\gamma u')(\alpha) = 0.
\end{aligned}
\right.
$$
In both cases, we have $u(\alpha) = 0$ and $\phi_{\lambda}(\alpha) (\gamma u')(\alpha) = 0$, and the same goes for $\beta$. We deduce from \eqref{eq:ipp} that 
$$
\displaystyle \int_C F \phi_{\lambda} = 0.
$$
Now suppose that $C \cap \{0, \pi\} = \emptyset$. In that case, $u(\alpha)=u(\beta)=(\gamma u')(\alpha)=(\gamma u')(\beta)=0$. Hence
$$
\displaystyle \int_C F \psi_{\lambda} =0.
$$
Then we have shown in every case that $\M_{\lambda} (F, \omega) = 0$. \\
$(\Leftarrow)$ Suppose that $\M_{\lambda} (F, \omega) = 0$. We can sum the integrals and get $\displaystyle \int_{\overline{(0,\pi) \setminus \omega}} F \phi_{\lambda} = 0$. Moreover, since $F=0$ in $\omega$, we also have $\displaystyle \int_{(0, \pi)} F \phi_{\lambda} = 0$. This orthogonality condition implies the existence of at least one solution $u_0$ to the equation 
$$
\mathfrak{L}u_0 - \lambda u_0 = F.
$$
We even have that every solution $u$ of this problem can be written under the form $u = u_0 + \mu \phi_{\lambda}$. We show in the rest of the proof that we can choose $\mu \in \R$ such that $u = 0$ in $\omega$.
First, we show that $\mu \in \R$ can be chosen in such a way that there exists a point $x_0 \in \overline{\omega}$ such that
$$
u(x_0)=(\gamma u')=0.
$$
Assume that $\overline{\omega} \cap \{0,\pi\}$, for example $0 \in \overline{\omega}$. We already have $u(0)=0$, so we only need to impose $(\gamma u')(0)=0$, that is
$$
(\gamma u_0 ')(0) + \mu (\gamma \phi_{\lambda}')(0) = 0.
$$
Since $(\gamma \phi_{\lambda} ')(0) \neq 0$, we can take such a $\mu$. \\
Now assume that $\overline{\omega} \cap (0, \pi) = \emptyset$. We denote $[0,\beta]$ the connected component of $\overline{(0, \pi) \setminus \omega}$ that contains $0$. We know
$$
\displaystyle \int_0^{\beta} F \phi_{\lambda} = 0.
$$
Since $F=0$ in $\omega$, $\beta$ can be replaced by $\beta + \delta$ with $\delta$ small enough so that $(\beta, \beta + \delta) \subseteq \omega$, with $\phi_{\lambda} (\beta + \delta) \neq 0$. Then we can set $\mu$ such that 
$$
u(\beta + \delta) = u_0(\beta + \delta) + \mu \phi_{\lambda} (\beta + \delta) = 0.
$$
Then we have 
$$
0 = \displaystyle \int_0^{\beta + \delta} F \phi_{\lambda} = -[(\gamma u')(\beta + \delta)\phi_{\lambda}(\beta + \delta) - u(\beta + \delta) (\gamma \phi_{\lambda}')(\beta + \delta)] + \underbrace{[(\gamma u')(0)\phi_{\lambda}(0) - u(0) (\gamma \phi_{\lambda}')(0)]}_{=0}.
$$
Since $u(\beta + \delta) = 0$ and $\phi_{\lambda}(\beta + \delta) \neq 0$, we get
$$
(\gamma u')(\beta + \delta)=0.
$$
Hence, $u$ and $\gamma u'$ vanish at the same point $x_0 = \beta + \delta$ in $\omega$. \\
Now, we want to show that $u = 0$ in $\omega$. Assume that there exists $x_1 \in \omega$ such that $u(x_1) \neq 0$. Without loss of generality, assume that $x_0 < x_1$. Then $[x_0, x_1] \cap \overline{(0, \pi) \setminus \omega}$ is a reunion of connected components of $\overline{(0, \pi) \setminus \omega}$, and none of those components touch $\{0, \pi\}$. \\
Since $F=0$ in $\omega$, and $\M_{\lambda}(F,\omega) = 0$, we get
$$
\displaystyle \int_{x_0}^{x_1} F \phi_{\lambda} = \displaystyle \int_{x_0}^{x_1} F \psi_{\lambda} = 0.
$$
Then we get the system
\begin{equation} \label{eq:syswro}
    \left \{
    \begin{aligned}
        0 &= -(\gamma u')(x_1) \phi_{\lambda} (x_1) + u(x_1)(\gamma \phi_{\lambda}')(x_1) \\
        0 &= -(\gamma u')(x_1) \psi_{\lambda} (x_1) + u(x_1)(\gamma \psi_{\lambda}')(x_1).
    \end{aligned}
    \right.
\end{equation}
Since $\phi_{\lambda}$ and $\psi_{\lambda}$ are two independant solutions of the same ODE of order two, the wronskian matrix
$$
\begin{pmatrix}
    \phi_{\lambda}(x_1) & -(\gamma \phi_{\lambda}')(x_1) \\
    \psi_{\lambda}(x_1) & -(\gamma \psi_{\lambda}')(x_1)
\end{pmatrix}
$$
is invertible. Consequently, \eqref{eq:syswro} has only one solution possible, that is
$$
u(x_1) = (\gamma u')(x_1) = 0.
$$
This is a contradiction, so $u=0$ in $\omega$, and this ends the proof.
\end{proof}

We will also use the following result about the existence and properties of biorthogonal families to a family of exponential functions, proved in \cite{coursboyer} (Theorem V.4.16).

\begin{theorem} \label{th:biortho}
    Let $\tau >0$. The family $\{e_{0,\lambda}: t \mapsto e^{-\lambda t}, e_{1,\lambda}: t \mapsto te^{-\lambda t}\}_{\lambda \in \Lambda}$ admits a biorthogonal family $\{r_{0,\lambda}, r_{1,\lambda}\}_{\lambda \in \Lambda}$ of functions in $L^2(0,\tau)$, that is
    \begin{equation} \label{eq:biortho}
        \displaystyle \int_0^\tau e_{i,\lambda}(t) r_{j,\mu}(t) dt = \delta_{(i,\lambda),(j,\mu)}. 
    \end{equation}
    Moreover, we have the following estimate
    $$
    \Vert r_{i,\lambda} \Vert_{L^2(0,\tau)} \leq K e^{\tau \frac{\lambda}{2} + K \sqrt{\lambda} + \frac{K}{\tau}}, \: \forall \lambda \in \Lambda, \: \forall i \in \{0,1\}.
    $$
    where $K>0$ does not depend on $\lambda$.
\end{theorem}

\section{Approximate controllability}
\subsection{Already existing results with no order one coupling term}
In \cite{boyer}, the authors have already proved some results when there is no order one term, that is $p=0$. Since we will use similar ideas to treat our case with this term, let us introduce those results.
\vspace{1em}

First, let us rewrite the problem we are interested in in this section:

\begin{equation} \label{eq:noP}
\left \{
\begin{aligned}
    \partial_{t} y + \L y &= \begin{pmatrix} 0 & 0 \\ -q(x) & 0 \end{pmatrix}y + \mathbb{1}_{\omega} \begin{pmatrix}
        v \\ 0
    \end{pmatrix} \\
    y(0,\cdot) = y(\pi,\cdot) &= 0 \\
    y(\cdot, 0) &= y^{0}
\end{aligned}
\right.
\end{equation}

We have the following characterization of its approximate controllability.

\begin{theorem} \label{th:ACnoP}
    Assume that $\supp q \cap \omega = \emptyset$. Then system \eqref{eq:noP} is approximately controllable if and only if
    $$
    \forall \lambda \in \Lambda, \: \textnormal{rank} \: \M_{\lambda}(-q \phi_{\lambda}, \omega) = 1.
    $$
\end{theorem}

\begin{remark}
    In this formula, the rank condition is understood in the vector space $(\R^{2})^{\mathcal{C} \left( \overline{(0,\pi) \setminus \omega}\right)}$.
\end{remark}
\begin{proof}
    We use the criterion given by Theorem \ref{th:fh}. Hence, the system is approximately controllable if and only if
    $$
    \forall s \in \C, \forall u \in \mathcal{D}(\A^*), \:
    \left.
        \begin{aligned}
             (\A^*-s) u &=0 &\textnormal{in} \: (0, \pi)  \\
             B^{*} u &=0 &\textnormal{in} \: \omega
        \end{aligned}
    \right \} \Rightarrow u = 0.
    $$
    The only non-trivial case is the one where $s=\lambda$ where $\lambda \in \Lambda$ is an eigenvalue of $\A^*$. Then $u$ is a solution of $\L u-A(x)^*u=\lambda u$, so it can be written
    $$
    u = \begin{pmatrix}
        u_{1} \\
        \delta \phi_{\lambda}
    \end{pmatrix}
    $$
    with $\delta \in \R$ and $u_1 \in \mathcal{D}(\mathfrak{L})$ such that
    $$
    \mathfrak{L} u_1 - \lambda u_1 = -\delta q \phi_{\lambda}.
    $$
    Applying Theorem \ref{th:uc} with $F=-\delta q \phi_{\lambda}$, and since we assumed $q \mathbb{1}_{\omega} = 0$, such a solution $u$ exists and satisfies $u_1 = 0$ in $\omega$ if and only if
    \begin{equation} \label{rang}
        \M_{\lambda}(-\delta q \phi_{\lambda}, \omega) = 0.
    \end{equation}
    Moreover, $u=0$ if and only if $\delta = 0$ and $u_1 = 0$ in $\omega$. This follows from a unique continuation property on a simple parabolic equation. \\
    Finally, the Fattorini-Hautus test is satisfied if and only if \eqref{rang} implies $\delta =0$. Note that 
    $$
    \M_{\lambda}(-\delta q \phi_{\lambda}, \omega) = \delta \M_{\lambda}(-q\phi_{\lambda}, \omega).
    $$
    Then, the theorem is proved, for \eqref{rang} implies $\delta =0$ if and only if $\M_{\lambda}(-q\phi_{\lambda}, \omega) \neq 0$, that is exactly rank $\M_{\lambda}(-q\phi_{\lambda}, \omega) = 1$.
\end{proof}

\subsection{Back to our system}

In the first section we claimed that Theorem \ref{th:1} from \cite{duprez} cannot be true for we can find cases where the hypotheses are satisfied and yet system \eqref{eq:D} is not approximately controllable. In this part, we state a slightly different theorem about the approximate controllability, and exhibit a particular example where Theorem \ref{th:1} fails. 
\subsubsection{A necessary and sufficient condition}
We denote $\Tilde{\Lambda} = \{ \lambda \in \Lambda$, $-q \phi_{\lambda} + \partial_{x}(p \phi_{\lambda}) = 0$ in $\omega \}$.

\begin{theorem} \label{th:AC}
    System \eqref{eq:D} is approximately controllable if and only if 
    \begin{equation} \label{eq:AC}
        \forall \lambda \in \Tilde{\Lambda}, \:
        \textnormal{rank} \: \M_{\lambda}(-q\phi_{\lambda} + \partial_{x}(p\phi_{\lambda}), \omega) = 1.
    \end{equation}
\end{theorem}

\begin{proof}
    As for the previous theorem, we use the criterion of Fattorini-Hautus (Theorem \ref{th:fh}) and the unique continuation property (Theorem \ref{th:uc}) to prove this result. Indeed, from Theorem \ref{th:fh}, we know that system \eqref{eq:D} is approximately controllable if and only if for any $s\in \C$ and any $u \in \mathcal{D}(\A^*)$
    \eqref{eq:fh} is satisfied. This is equivalent to the following condition
    $$
    \forall \lambda \in \Lambda, \: \left \{ 
    \begin{aligned}
        \mathfrak{L} u - \lambda u &= -q\phi_{\lambda} + \partial_x (p\phi_{\lambda}) \\
        u &= 0 \textnormal{ in } \omega
    \end{aligned}
    \right.
    \textnormal{ has no solution in } \mathcal{D}(\mathfrak{L}).
    $$
    Using Theorem \ref{th:uc} we get that the system is approximately controllable if and only if
    \begin{equation}
        \forall \lambda \in \Lambda, \: \bigg \{
        -q\phi_{\lambda} + \partial_x (p\phi_{\lambda}) \neq 0 \textnormal{ in } \omega \textnormal{, or } \M_{\lambda}(-q\phi_{\lambda} + \partial_{x}(p\phi_{\lambda}), \omega) \neq 0 \bigg \}.
    \end{equation}
    It is clear that \eqref{eq:AC} is equivalent to this condition since for any $\lambda \in \Lambda$    
    $$
    \M_{\lambda}(-q\phi_{\lambda} + \partial_{x}(p\phi_{\lambda}), \omega) \neq 0 \Leftrightarrow \textnormal{rank} \: \M_{\lambda}(-q\phi_{\lambda} + \partial_{x}(p\phi_{\lambda}), \omega) = 1.
    $$
\end{proof}
\begin{remark}
    In the case where the control domain is an interval $\omega = (a,b)$, this rank condition simply rewrites as a condition on the two following integrals:
    $$
    \left \vert \displaystyle \int_0^a \left(q - \frac{\partial_{x}p}{2} \right) \phi_{\lambda}^2 \right \vert  + \left \vert \displaystyle \int_0^\pi \left(q - \frac{\partial_{x}p}{2} \right) \phi_{\lambda}^2 \right \vert \neq 0.
    $$
    This case is detailed in \cite{duprez} (see Theorem 1.2).
\end{remark}

\begin{example}
    Let us explicit some simple cases of this theorem: 
    \begin{itemize}
        \item If $p=0$, we have $\Tilde{\Lambda} = \{ \lambda \in \Lambda$, $q \phi_{\lambda}  = 0$ in $\omega \}$. 
            \begin{itemize}
                \item If $\supp q \cap \omega \neq \emptyset$ then $\Tilde{\Lambda} = \emptyset$, so there is nothing to check in \eqref{eq:AC}, and the system is approximately controllable.
                \item If $\supp q \cap \omega = \emptyset$ then $\Tilde{\Lambda} = \Lambda$, so condition \eqref{eq:AC} becomes 
                \begin{equation}
                    \forall \lambda \in \Lambda , \: \textnormal{rank} \: \M_{\lambda}(q\phi_{\lambda} , \omega) = 1.
                \end{equation}
                Then we retrieve the well known case with only an order zero coupling term given in Theorem \ref{th:ACnoP}.
            \end{itemize}
        \item If $q=0$ we have $\Tilde{\Lambda} = \{ \lambda \in \Lambda$, $\partial_{x}(p\phi_{\lambda}) = 0$ in $\omega \}$. 
        \begin{itemize}
            \item If $\supp p \cap \omega = \emptyset$ then $\Tilde{\Lambda} = \Lambda$, and the system is approximately controllable if and only if 
            \begin{equation}
                \forall \lambda \in \Lambda , \: \textnormal{rank} \: \M_{\lambda}(\partial_x (p \phi_{\lambda}) , \omega) = 1.
            \end{equation}
            \item If $\supp p \cap \omega \neq \emptyset$; we treat two cases:
            \begin{itemize}
                \item In the particular case where $\omega$ is an interval $(a,b)$, if there exists $x_0 \in \overline{\omega}$ such that $p(x_0)=0$, then we have approximate controllability for the system. Indeed, with the same arguments as previously (Fattorini-Hautus criterion and unique continuation), the system is not approximately controllable if and only if
                $$
                \left \{
                \begin{aligned}
                    \partial_{x} (p\phi_{\lambda}) &= 0 &\textnormal{in} \: \omega \\
                    \displaystyle \int_{0}^{a}\partial_{x}(p\phi_{\lambda})\phi_{\lambda} &= \displaystyle \int_{b}^{\pi}\partial_{x}(p\phi_{\lambda})\phi_{\lambda} = 0.
                \end{aligned}
                \right.
                $$
                Only using the first condition on $\partial_{x} (p\phi_{\lambda})$ in $\omega$, and since $p\phi_{\lambda}$ is continuous on $\overline{\omega}$, we deduce that $p\phi_{\lambda}$ is constant in $\overline{\omega}$. Hence 
                \begin{equation} \label{eq:cst}
                (p\phi_{\lambda})(x) = (p\phi_{\lambda})(x_0) = 0, \: \forall x \in \overline{\omega} 
                \end{equation}
                since $p(x_0)=0$. However, as $\phi_{\lambda}$ is an eigenfunction of $\A^*$, its zeroes are isolated, so \eqref{eq:cst} implies that $p(x)=0, \: \forall x \in \overline{\omega}$. We assumed $\supp p \cap \omega \neq \emptyset$, so this is a contradiction. \\
                Then the system is necessarily approximately controllable. 
                \item If for all $x \in \overline{\omega}$, $p(x) \neq 0$, that is $\overline{\omega} \subseteq \supp p$,
                then $\#\Tilde{\Lambda} \leq 1$.
                Indeed, suppose that there exist $\lambda_0, \lambda_1 \in \Tilde{\Lambda}$. Then 
                $$
                \left \{
                \begin{aligned}
                    \partial_{x}(p \phi_{\lambda_0}) &= 0 \: \textnormal{in} \: \omega \\
                    \partial_{x}(p \phi_{\lambda_1}) &= 0 \: \textnormal{in} \: \omega.
                \end{aligned}
                \right.
                $$
                that is, in $\omega$
                $$
                \left \{
                \begin{aligned}
                    \phi_{\lambda_0}' &= -\frac{p'}{p} \phi_{\lambda_0} \\
                    \phi_{\lambda_1}' &= -\frac{p'}{p} \phi_{\lambda_1}.
                \end{aligned}
                \right.
                $$
                Then $\phi_{\lambda_0}$ and $\phi_{\lambda_1}$ are two solutions of the same ODE of order one, and we conclude that $\phi_{\lambda_0}$ and $\phi_{\lambda_1}$ are proportional in $\omega$, that is $\lambda_0 = \lambda_1$. Hence $\# \Tilde{\Lambda} \leq 1$.
                \begin{itemize}
                    \item If $\#\Tilde{\Lambda} = 0$, then $\Tilde{\Lambda} = \emptyset$, thus the system is approximately controllable.
                    \item And if $\#\Tilde{\Lambda} = 1$, then the system may not be approximately controllable. We give details on this case in the following section.
                \end{itemize}
            \end{itemize}
        \end{itemize}
    \end{itemize}
\end{example}
\subsubsection{A particular case} \label{counterex}

We will prove the following result, that gives a counter-example to Theorem \ref{th:1}.

\begin{corollary}
    Suppose that $q=0$. There exists a $p \in W^{1, \infty}(0,\pi)$ and a control domain $\omega$ such that $\supp p \cap \omega \neq \emptyset$ and system \eqref{eq:D} is not approximately controllable.
\end{corollary}

\begin{proof}
    Let us rewrite system \eqref{eq:D} in the case we consider here:
    \begin{equation} \label{eq:DnoQ}
        \left \{
        \begin{aligned}
          \partial_{t} y + \L y &= \begin{pmatrix} 0 & 0 \\ -p(x) & 0 \end{pmatrix} \partial_{x}y + \mathbb{1}_{\omega} \begin{pmatrix}
              v \\ 0
          \end{pmatrix} \\
            y(0,\cdot) = y(\pi,\cdot) &= 0 \\
            y(\cdot, 0) &= y^{0}
        \end{aligned}
        \right.
    \end{equation}
    We want to use the unique continuation property given by Theorem \ref{th:uc}, so we look for a $p$ and a control domain $\omega =(a,b)$ such that for some $\lambda \in \Lambda$
    $$
    \left\{
    \begin{aligned}
        \partial_{x} (p\phi_{\lambda}) &= 0 &\textnormal{in} \: \omega \\
        \displaystyle \int_{0}^{a}\partial_{x}(p\phi_{\lambda})\phi_{\lambda} &= \displaystyle \int_{b}^{\pi}\partial_{x}(p\phi_{\lambda})\phi_{\lambda} = 0
    \end{aligned}
    \right.
    $$
    Indeed, if it is possible to find such $p$ and $\omega$, then by applying Theorem \ref{th:fh}, we can conclude that the system is not approximately controllable.
    We define, for a fixed $\lambda \in \Lambda$,
    $$
    p = \left \{
    \begin{array}{ccc}
        p_0 + \theta_{a} \quad &\textnormal{in}& \quad (0,a) \\
        \displaystyle \frac{1}{\phi_\lambda} \quad &\textnormal{in}& \quad \omega \\
        p_0 + \theta_{b} \quad &\textnormal{in}& \quad (b,\pi)
        \end{array}
        \right.
    $$ 
    where $\omega$ is such that $\phi_{\lambda}$ does not vanish in $\omega$, and $p_0$, $\theta_a$, $\theta_b$ are chosen functions (see Figure \ref{fig:contrex1}) such that:
    \begin{itemize}
        \item $p_0$ is such that both the junctions with $\displaystyle \frac{1}{\phi_{\lambda}}$ in $a$ and $b$ are smooth.
        \item $\theta_a$ is decreasing in $(0,a)$ and such that $\theta_{a}(a)=0$.
        \item $\theta_b$ is increasing in $(b,\pi)$ and such that $\theta_{b}(b)=0$.
        \end{itemize} 
    \begin{figure}[ht]
        \begin{center}
\begin{tikzpicture}
\begin{axis}[
    axis lines = center,
    axis y line = none,
    xtick = {0},
    extra x ticks = {0.25, 0.75, 1},
    extra x tick labels = {\(a\), \(b\), \(\pi\)},
]

\addplot [
    domain=0:1,
    samples=100,
    color=black,
    ]
    {0.1*sin(pi*deg(x))};
\addlegendentry{$\phi_{\lambda}$}

\addplot [
    domain=0.25:0.75,
    samples=100,
    color=purple,
    ]
    {0.03/(sin(pi*deg(x)))};
\addlegendentry{\(\frac{1}{\phi_{\lambda}}\)}

\addplot [
    domain=0:0.25, 
    samples=100, 
    color=red,
]
{(x-0.25)^2};
\addlegendentry{\(\theta_{a}\)}

\addplot [
    domain=0.75:1, 
    samples=100, 
    color=blue,
    ]
    {(x-0.75)^2};
\addlegendentry{\(\theta_{b}\)}

\addplot [
    domain=0:0.25,
    samples=100,
    color=green,
    ]
    {0.03/sin(pi*deg(x)) *x^2*(-128*x+48)};
\addlegendentry{\(p_0\)}

\addplot [
    domain=0.75:1,
    samples=100,
    color=green,
    ]
    {0.03/sin(pi*deg(1-x)) *(1-x)^2*(-128*(1-x)+48)};
\end{axis}
\end{tikzpicture}
\end{center}
        \caption{Example of choices for $p_0$, $\theta_a$ and $\theta_b$}
            \label{fig:contrex1}
        \end{figure}
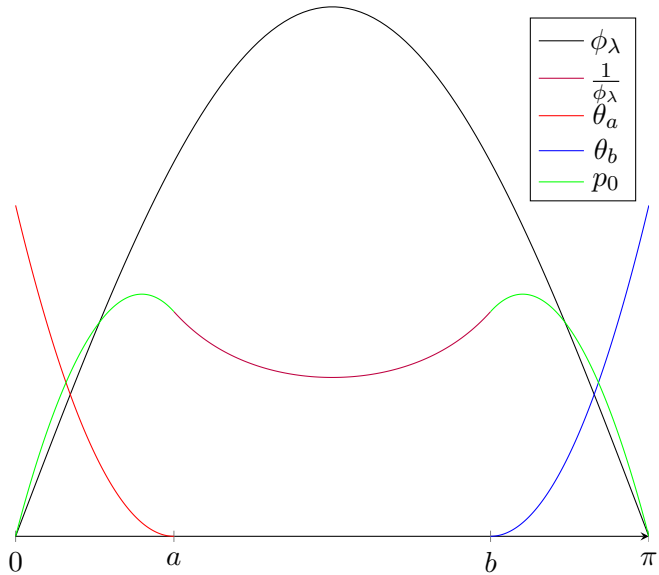 
    Then, it is clear that $\partial_{x}(p\phi_{\lambda}) = 0$ in $\omega$, and we have
    \begin{equation} \label{ippa}
        \int_{0}^{a}\partial_{x}(p\phi_{\lambda})\phi_{\lambda} = \underbrace{\int_{0}^{a}\partial_{x}(p_0\phi_{\lambda})\phi_{\lambda}}_{\coloneqq \alpha_a} + \underbrace{\int_{0}^{a}\partial_{x}(\theta_a\phi_{\lambda})\phi_{\lambda}}_{\coloneqq\beta_a} \\
    \end{equation}
    and
    \begin{equation} \label{ippb}
        \int_{b}^{\pi}\partial_{x}(p\phi_{\lambda})\phi_{\lambda} = \underbrace{\int_{b}^{\pi}\partial_{x}(p_0\phi_{\lambda})\phi_{\lambda}}_{\coloneqq \alpha_b} + \underbrace{\int_{b}^{\pi}\partial_{x}(\theta_b\phi_{\lambda})\phi_{\lambda}}_{\coloneqq\beta_b}.
    \end{equation}
    Integrating by parts the second term of the right-hand side \eqref{ippa}, we get
    $$
    \beta_a = \int_{0}^{a}\partial_{x}(\theta_a\phi_{\lambda})\phi_{\lambda} = \int_{0}^{a}\frac{\partial_{x}\theta_a}{2}\phi_{\lambda}^{2}.
    $$
    That is $\beta_a < 0$; in particular $\beta_a \neq 0$. Then we have, by rescaling our $\theta_a$ if needed, the following
    $$
    \int_{0}^{a}\partial_{x}(p\phi_{\lambda})\phi_{\lambda} = 0. $$
    Doing the exact same thing with the first term in \eqref{ippb} we also get that 
    $$
    \int_{b}^{\pi}\partial_{x}(p\phi_{\lambda})\phi_{\lambda} = 0.
    $$
    Finally, we found a function $p$ and a control domain $\omega$ such that the Fattorini-Hautus test is not verified. 
Since this function satisfies $(\supp p)\cap \omega \neq \emptyset$, this case is indeed a counter-example to the result claimed in \cite{duprez} (Theorem \ref{th:1} here).
\end{proof}

\section{Null controllability}
In this part, we will use the fact that $\Lambda$ is countable, and rearrange the eigenvalues in an increasing way. Indeed, now we write $\Lambda = \{ \lambda_k \}_{k \in \N}$. 
\subsection{Known results without order one coupling term}
In this report, we treat the case where the system contains a coupling term of order one, that we denoted $p$. In the literature, the case without this term, that is only with the order zero coupling term $q$, has already been treated, and almost extensively (see \cite{boyermorancey} and \cite{boyer}). Let us present some results about this particular case, that we will generalise to the problem with the order one term.
\vspace{1em}

First, we need to define a few quantities and notions. \begin{definition}
    For $F \in L^2(0,\pi, \R)$, $k \in \N$ we define:
    \begin{itemize}
        \item $\mathcal{N}_{k,1}(F,\omega)= \sup \left \{ \left \vert \displaystyle \int_{\mathcal{C}} F \phi_{\lambda_k} \right \vert, \: \mathcal{C} \textnormal{ connected component of } \overline{(0,\pi)\setminus \omega} \right \}$
        \item $\mathcal{N}_{k,2}(F,\omega)= \sup \left \{ \left \vert \displaystyle \int_{\mathcal{C}} F \psi_{\lambda_k} \right \vert, \: \mathcal{C} \textnormal{ connected component of } \overline{(0,\pi)\setminus \omega}, \mathcal{C} \cap \{0, \pi \} = \emptyset \right \}$
    \end{itemize}    
    and finally $\mathcal{N}_{k}(F,\omega)= \max \{\mathcal{N}_{k,1}(F,\omega), \mathcal{N}_{k,2}(F,\omega)\}. $
\end{definition}
\begin{definition}
    We call \textit{minimal null control time} a time $T_0 \geq 0$ such that
    \begin{itemize}
        \item for all $T>T_0$, System \eqref{eq:D} is null controllable.
        \item for all $T<T_0$, System \eqref{eq:D} is not null controllable.
    \end{itemize}
\end{definition}

In \cite{boyermorancey}, the following result about the minimal null control time is stated.

\begin{theorem}
    Suppose $\supp q \cap \omega = \emptyset$ and $\mathcal{N}_{k}((I_{\lambda_k}(q)-q)\phi_{\lambda_k},\omega) \neq 0 \: \forall k \in \N$ (that is exactly the characterization of the approximate controllability). Then the minimal null control time is given by
    \begin{equation} \label{eq:NCtime}
        T_0 = \limsup_{k \rightarrow \infty}{\frac{-\ln{\mathcal{N}_k((I_{\lambda_k}(q)-q)\phi_{\lambda_k},\omega)}}{\lambda_k}}.
    \end{equation}
\end{theorem}

\subsection{Moments method with coupling term of order one}

In this part, we will prove the following theorem, that gives exactly the minimal null control time for our problem \eqref{eq:D}.

\begin{theorem}
    Suppose $(\supp p \cup \supp q)\cap \omega = \emptyset$, and that the system is approximately controllable, that is $\mathcal{N}_k(I_{\lambda_k}(p,q) + \partial_x(p\phi_{\lambda_k})-q\phi_{\lambda_k},\omega) \neq 0 \: \forall k \in \N$. The minimal null control time is given by
    \begin{equation}
        T_0 = \limsup_{k \rightarrow \infty} {\frac{ - \ln{\mathcal{N}_k(I_{\lambda_k}(p,q) \phi_{\lambda_k} + \partial_x(p\phi_{\lambda_k})-q\phi_{\lambda_k},\omega)}}{\lambda_k}}.
    \end{equation}
\end{theorem}
\begin{proof} The proof will be divided in two parts. First, we will show that for all $T>T_0$, the system is null controllable. Then, we will show that this time $T_0$ is minimal.
\begin{itemize}
\item Take $T>0$. We will use a moments method. \\
Let $y^0=(y^0_1, y^0_2) \in L^2(0,\pi)^2$, and take $\xi_F \in \mathcal{B}\coloneqq \{\Phi_{\lambda_k}, \Psi_{\lambda_k}\}_{k \in \N}$ in the dual problem \eqref{eq:dual} . We integrate by parts and use the system:
$$
\begin{aligned}
    (y(T), \xi(T))_{L^2(0,\pi)^2} - (y^0, \xi(0))_{L^2(0,\pi)^2} &= \displaystyle \int_{0}^{T} \partial_{t} \left(\displaystyle \int_{0}^{\pi} (y_1 \xi_1 + y_2 \xi_2) dx \right) dt \\
    &= \displaystyle \int_{0}^{T} \displaystyle \int_0^\pi (\partial_t y_1 \xi_1 + y_1 \partial_t \xi_1 + \partial_t y_2 \xi_2 + y_2 \partial_t \xi_2) dx dt \\
    &= \displaystyle \int_{0}^{T} \displaystyle \int_0^\pi \underbrace{(\partial_t y_1 - \mathfrak{L}y_1)}_{=\mathbb{1}_{\omega} v} \xi_1 dx dt \\
    &+ \displaystyle \int_{0}^{T} \displaystyle \int_0^\pi \underbrace{(\partial_t y_2 - \mathfrak{L} y_2 + p \partial_x y_1 + qy_1)}_{=0}\xi_2 dx dt \\
    &= \displaystyle \int_{0}^{T} \displaystyle \int_0^\pi \mathbb{1}_{\omega} v \xi_1 dx dt \\
    &= \displaystyle \int_{0}^{T} \displaystyle \int_0^\pi v \mathbb{1}_{\omega} B^* \xi dx dt.
\end{aligned}
$$
Since $\mathcal{B}$ is a complete family in $L^2(0,\pi)^2$, system \eqref{eq:D} is null controllable if and only if for every $y_0 \in L^2(0,\pi)^2$ there exists a control $v \in L^2((0,\pi) \times (0,T))$ such that for all $\xi_F \in \mathcal{B}$ the solution $y$ of system \eqref{eq:D} satisfies 
\begin{equation} \label{eq:momentpb}
    \displaystyle \int_{0}^{T} \displaystyle \int_0^\pi v \mathbb{1}_{\omega} B^* \xi dx dt =- (y^0, \xi(0))_{L^2(0,\pi)^2}
\end{equation}
where $\xi$ is the solution to the dual problem associated with the data $\xi_F$, that is $\xi = e^{-(T-t)\A^*}\xi_F$.
\vspace{1em}

Now, we look for a control $v$ such that \eqref{eq:momentpb} is true. We search it of the following form
\begin{equation} \label{eq:control}
    \begin{aligned}
        v(t,x) = &\sum_{\mu \in \Lambda_1} \left( \alpha_{\mu} r_{\mu,0}(t)f_{\mu}(x) + \beta_{\mu} r_{\mu,1}(t) g_{\mu}(x) \right) \\
        + &\sum_{\mu \in \Lambda_2} r_{\mu, 0} (t) [\gamma_{\mu} h_{\mu}(x) + \theta_{\mu} \Tilde{h}_{\mu}(x)]
    \end{aligned}
\end{equation}

where $f_{\mu}, g_{\mu}, h_{\mu}, \Tilde{h}_{\mu}$ are functions in $L^2(0,\pi)$ to be determined such that their support is included in $\omega$, and $\alpha_{\mu}, \beta_{\mu}, \gamma_{\mu}, \theta_{\mu} \in \R$.

Let us suppose that such a control exists, and explicit the form of the different parameters. We distinguish two cases depending on the geometrical multiplicity of the eigenvalue $\lambda_k$.

\begin{itemize}
    \item First case: $\lambda_k \in \Lambda_1$ is a simple eigenvalue, that is $I_{\lambda_k}(p,q) \neq 0$. 

    We recall that Proposition \ref{prop:spec} gives that $\Phi_{\lambda_k}$ is an eigenvector and $\Psi_{\lambda_k}$ is a generalized eigenvector such that:
    $$
    \left \{
    \begin{aligned}
        (\A^* - \lambda_k)\Phi_{\lambda_k} &= 0 \\
        (\A^* - \lambda_k)\Psi_{\lambda_k} &= I_{\lambda_k}(p,q) \Phi_{\lambda_k}.
    \end{aligned}
    \right.
    $$
    Thus, we can explicit $\xi$:
    \begin{itemize}
        \item if $\xi_F = \Phi_{\lambda_k}$, then 
        \begin{equation}
            e^{-(T-t)\A^*} \Phi_{\lambda_k} = e^{-(T-t)\lambda_k} \Phi_{\lambda_k}.
        \end{equation}
        \item and if $\xi_F = \Psi_{\lambda_k}$, then 
        \begin{equation} \label{eq:mompsi}
            e^{-(T-t)\A^*} \Psi_{\lambda_k} = e^{-(T-t)\lambda_k} \left[\Psi_{\lambda_k} -(T-t) I_{\lambda_k}(p,q)\Phi_{\lambda_k} \right].
        \end{equation}
    \end{itemize}
    For such $\lambda_k$, the equation in the moment problem rewrites:
    \begin{equation} \label{eq:momphi}
        \displaystyle \int_0^T \left( v(t,\cdot), B^* e^{-(T-t)\lambda_k} \Phi_{\lambda_k} \right)_{L^2(\omega)} dt = - (y^0, e^{-T \lambda_k}\Phi_{\lambda_k})_{L^2(0,\pi)}
    \end{equation}
    and
    \begin{dmath} \label{eq:mompsi2}
        \displaystyle \int_0^T \left( v(t,\cdot), B^* e^{-(T-t)\lambda_k} \left[\Psi_{\lambda_k} -(T-t) I_{\lambda_k}(p,q)\Phi_{\lambda_k} \right] \right)_{L^2(\omega)} dt = - (y^0, e^{-T \lambda_k}(\Psi_{\lambda_k} - T I_{\lambda_k}(p,q)\Phi_{\lambda_k}))_{L^2(0,\pi)}.
    \end{dmath}
    We compute the left-hand side of \eqref{eq:momphi} with $v$ of the form given in \eqref{eq:control}. Note that the biorthogonality condition \eqref{eq:biortho} gives that the only remaining term in the sum given in \eqref{eq:control} is the one with $\lambda_k$. We thus get
    \begin{dmath*}
            \displaystyle \int_0^T  \left( v(t,\cdot), B^* e^{-(T-t)\lambda_k} \Phi_{\lambda_k} \right)_{L^2(\omega)} dt
            = \alpha_{\lambda_k}\left(f_{\lambda_k}, B^*\Phi_{\lambda_k}\right)_{L^2(\omega)}.
    \end{dmath*}
    The same kind of computation gives, for the left-hand side of \eqref{eq:mompsi}:
    \begin{dmath*}
        \displaystyle \int_0^T  \left( v(t,\cdot), B^* e^{-(T-t)\lambda_k} \left[\Psi_{\lambda_k} -(T-t) I_{\lambda_k}(p,q)\Phi_{\lambda_k} \right] \right)_{L^2(\omega)} dt
        = \alpha_{\lambda_k}\left(f_{\lambda_k}, B^*\Psi_{\lambda_k}\right)_{L^2(\omega)} + \beta_{\lambda_k} I_{\lambda_k}(p,q) \left(g_{\lambda_k}, B^* \Phi_{\lambda_k} \right)_{L^2(\omega)}.
    \end{dmath*}
    Thus, our moment problem reduces to finding $\alpha_{\lambda_k}, \beta_{\lambda_k}, f_{\lambda_k}$ and $g_{\lambda_k}$ such that
    $$
    \left \{
    \begin{aligned}
        \alpha_{\lambda_k}\left(f_{\lambda_k}, B^*\Phi_{\lambda_k}\right)_{L^2(\omega)} &= -e^{-\lambda_k T}(y^0, \Phi_{\lambda_k})_{L^2(0,\pi)} \\
        \alpha_{\lambda_k}\left(f_{\lambda_k}, B^*\Psi_{\lambda_k}\right)_{L^2(\omega)} + \beta_{\lambda_k} I_{\lambda_k}(p,q)\left(g_{\lambda_k}, B^* \Phi_{\lambda_k} \right)_{L^2(\omega)} &= -e^{-\lambda_k T} (y^0,\Psi_{\lambda_k}-T\Phi_{\lambda_k})_{L^2(0,\pi)} .
    \end{aligned}
    \right.
$$
    Let us take $f_{\lambda_k} = g_{\lambda_k} = B^* \mathbb{1}_{\omega}\Phi_{\lambda_k}$:
    $$
    \left \{
    \begin{aligned}
        \alpha_{\lambda_k} \Vert B^* \Phi_{\lambda_k} \Vert^{2}_{L^2(\omega)} &= -e^{-\lambda_k T}(y^0, \Phi_{\lambda_k})_{L^2(0,\pi)} \\
        \alpha_{\lambda_k}\left(B^* \Phi_{\lambda_k}, B^*\Psi_{\lambda_k}\right)_{L^2(\omega)} + \beta_{\lambda_k} I_{\lambda_k}(p,q) \Vert B^* \Phi_{\lambda_k} \Vert^{2}_{L^2(\omega)} &= -e^{-\lambda_k T} (y^0,\Psi_{\lambda_k}-TI_{\lambda_k}(p,q)\Phi_{\lambda_k})_{L^2(0,\pi)}.
    \end{aligned}
    \right.
    $$
    Since $I_{\lambda_k}(p,q)\Vert B^* \Phi_{\lambda_k} \Vert^{2}_{L^2(\omega)} \neq 0$, we can solve and get:
    $$
    \left \{
    \begin{aligned}
        \alpha_{\lambda_k} &= \frac{-e^{-\lambda_k T}(y^0, \Phi_{\lambda_k})_{L^2(0,\pi)}}{\Vert B^* \Phi_{\lambda_k} \Vert^{2}_{L^2(\omega)}} \\
        \beta_{\lambda_k} &= \frac{-e^{-\lambda_k T} (y^0,\Psi_{\lambda_k}-TI_{\lambda_k}(p,q)\Phi_{\lambda_k})_{L^2(0,\pi)}}{I_{\lambda_k}(p,q)\Vert B^* \Phi_{\lambda_k} \Vert^{2}_{L^2(\omega)}} + \frac{e^{-\lambda_k T}(y^0, \Phi_{\lambda_k})_{L^2(0,\pi)}(B^*\Phi_{\lambda_k}, B^*\Psi_{\lambda_k})_{L^2(\omega)}}{I_{\lambda_k}(p,q)\Vert B^* \Phi_{\lambda_k} \Vert^{4}_{L^2(\omega)}}.
    \end{aligned}
    \right.
    $$
    
    \item Second case: $\lambda_k \in \Lambda_2$ is a double eigenvalue, that is $I_{\lambda_k}(p,q) = 0$. 

    Proposition \ref{prop:spec} now gives that $\Phi_{\lambda_k}$ and $\Psi_{\lambda_k}$ are two independant eigenvectors:
    $$
    \left \{
    \begin{aligned}
        (\A^* - \lambda_k)\Phi_{\lambda_k} &= 0 \\
        (\A^* - \lambda_k) \Psi_{\lambda_k} &= 0.
    \end{aligned}
    \right.
    $$    
    Hence we get:
    \begin{equation}
            e^{-(T-t)\A^*} \Phi_{\lambda_k} = e^{-(T-t)\lambda_k} \Phi_{\lambda_k}
    \end{equation}
    and 
    \begin{equation}
            e^{-(T-t)\A^*} \Psi_{\lambda_k} = e^{-(T-t)\lambda_k} \Psi_{\lambda_k}.
    \end{equation}
    The moment problem rewrites: for $\lambda_k \in \Lambda_2$,
    \begin{equation} \label{eq:momphi2}
        \displaystyle \int_0^T \left( v(t,\cdot), B^* e^{-(T-t)\lambda_k} \Phi_{\lambda_k} \right)_{L^2(\omega)} dt = - (y^0, e^{-T\lambda_k}\Phi_{\lambda_k})_{L^2(0,\pi)}
    \end{equation}
    and 
    \begin{equation} \label{eq:momphitilde}
        \displaystyle \int_0^T \left( v(t,\cdot), B^* e^{-(T-t)\lambda_k} \Psi_{\lambda_k} \right)_{L^2(\omega)} dt = - (y^0, e^{-T\lambda} \Psi_{\lambda_k})_{L^2(0,\pi)} = .
    \end{equation}
    As in the previous case, we compute the left-hand side of \eqref{eq:momphi2} with a control of the form given by \eqref{eq:control}. We have: 
    $$
    \begin{aligned}
        \sum_{\mu \in \Lambda_2} \bigg[ \gamma_{\mu} (h_{\mu}, B^* \Phi_{\lambda_k})_{L^2(\omega)} &\underbrace{\displaystyle \int_0^T e^{-(T-t)\lambda_k}r_{\mu,0}(t)dt}_{= \delta_{\mu,\lambda_k}}
        + \theta_{\mu} (\Tilde{h}_{\mu}, B^* \Phi_{\lambda_k})_{L^2(\omega)} \underbrace{\displaystyle \int_0^T e^{-(T-t)\lambda_k}r_{\mu,0}(t)dt}_{= \delta_{\mu,\lambda_k}} \bigg] \\
        &= \gamma_{\lambda_k}(h_{\lambda_k}, B^* \Phi_{\lambda_k})_{L^2(\omega)} + \theta_{\lambda_k}(\Tilde{h}_{\lambda_k}, B^*\Phi_{\lambda_k})_{L^2(\omega)}.
    \end{aligned}
    $$
    The left-hand side of \eqref{eq:momphitilde} becomes
    $$
    \begin{aligned}
        \sum_{\mu \in \Lambda_2} \bigg[ \gamma_{\mu} (h_{\mu}, B^* \Psi_{\lambda_k})_{L^2(\omega)} &\underbrace{\displaystyle \int_0^T e^{-(T-t)\lambda_k}r_{\mu,0}(t)dt}_{= \delta_{\mu,\lambda_k}} 
        + \theta_{\mu} (\Tilde{h}_{\mu}, B^* \Psi_{\lambda_k})_{L^2(\omega)} \underbrace{\displaystyle \int_0^T e^{-(T-t)\lambda_k}r_{\mu,0}(t)dt}_{= \delta_{\mu,\lambda_k}} \bigg] \\
        &= \gamma_{\lambda_k}(h_{\lambda_k}, B^* \Psi_{\lambda_k})_{L^2(\omega)} + \theta_{\lambda_k}(\Tilde{h}_{\lambda_k}, B^*\Psi_{\lambda_k})_{L^2(\omega)}.
    \end{aligned}
    $$

    The moment problem reduces then to finding $\gamma_{\lambda_k}, \theta_{\lambda_k}, h_{\lambda_k}$ and $\Tilde{h}_{\lambda_k}$ such that
    \begin{equation} \label{eq:mom2}
        \left\{
        \begin{aligned}
            \gamma_{\lambda_k} (h_{\lambda_k}, B^*\Phi_{\lambda_k})_{L^2(\omega)} + \theta_{\lambda_k}(\Tilde{h}_{\lambda_k}, B^* \Phi_{\lambda_k})_{L^2(\omega)} &= -e^{-\lambda_k T}(y^0, \Phi_{\lambda_k})_{L^2(0,\pi)} \\
            \gamma_{\lambda_k} (h_{\lambda_k}, B^*\Psi_{\lambda_k})_{L^2(\omega)} + \theta_{\lambda_k}(\Tilde{h}_{\lambda_k}, B^* \Psi_{\lambda_k})_{L^2(\omega)} &= -e^{-\lambda_k T}(y^0, \Psi_{\lambda_k})_{L^2(0,\pi)}
        \end{aligned}
        \right.
    \end{equation}
    If we find $\{h_{\lambda_k}, \Tilde{h}_{\lambda_k}\}$ a biorthogonal family to $\{B^* \Phi_{\lambda_k}, B^* \Psi_{\lambda_k}\}$ in $L^2(\omega)$, we will have:
    \begin{equation} \label{eq:mom2simpl}
        \eqref{eq:mom2} \Leftrightarrow \left \{ \begin{aligned}
            \gamma_{\lambda_k} &= -e^{-\lambda_k T}(y^0, \Phi_{\lambda_k})_{L^2(0,\pi)} \\
            \theta_{\lambda_k}&= -e^{-\lambda_k T}(y^0, \Psi_{\lambda_k})_{L^2(0,\pi)}.
        \end{aligned}
        \right.
    \end{equation}
    Hence we look for a family $\{h_{\lambda_k},\Tilde{h}_{\lambda_k}\}$ such that:
    \begin{equation}
        \left \{
        \begin{aligned}
            (h_{\lambda_k},B^* \Phi_{\lambda_k})_{L^2(\omega)} &= (\Tilde{h}_{\lambda_k},B^* \Psi_{\lambda_k})_{L^2(\omega)} = 1 \\
            (h_{\lambda_k},B^* \Psi_{\lambda_k})_{L^2(\omega)} &= (\Tilde{h}_{\lambda_k},B^* \Phi_{\lambda_k})_{L^2(\omega)} = 0.
        \end{aligned}
        \right.
    \end{equation}
    Take $\displaystyle h_{\lambda_k}\coloneqq \frac{B^* \mathbb{1}_{\omega} \Phi_{\lambda_k}}{\Vert B^* \Phi_{\lambda_k} \Vert^2_{L^2(\omega)}}$ and $\displaystyle \Tilde{h}_{\lambda_k}\coloneqq \frac{B^* \mathbb{1}_{\omega} \Psi_{\lambda_k}}{\Vert B^* \Psi_{\lambda_k} \Vert^2_{L^2(\omega)}}$.
    Then we can check that $(h_{\lambda_k}, B^* \Phi_{\lambda_k})_{L^2(\omega)}=(\Tilde{h}_{\lambda_k},B^* \Phi_{\lambda_k})_{L^2(\omega)}=1$, and
    $$
    (h_{\lambda_k}, B^* \Psi_{\lambda_k})_{L^2(\omega)} = \frac{(B^* \Phi_{\lambda_k}, B^* \Psi_{\lambda_k})_{L^2(\omega)}}{\Vert B^* \Phi_{\lambda_k} \Vert^2_{L^2(\omega)}}
    $$
    $$
    (\Tilde{h}_{\lambda_k}, B^* \Phi_{\lambda_k})_{L^2(\omega)} = \frac{(B^* \Phi_{\lambda_k}, B^* \Psi_{\lambda_k})_{L^2(\omega)}}{\Vert B^* \Phi_{\lambda_k} \Vert^2_{L^2(\omega)}}.
    $$
    By definition, $\Phi_{\lambda_k}= \begin{pmatrix}
        \phi_{\lambda_k} \\
        0
    \end{pmatrix}$ and $\Psi_{\lambda_k} = \begin{pmatrix}
        \psi_{\lambda_k} \\
        \phi_{\lambda_k}
    \end{pmatrix}$, so 
    $$
    (B^* \Phi_{\lambda_k}, B^* \Psi_{\lambda_k})_{L^2(\omega)}=(\phi_{\lambda_k}, \psi_{\lambda_k})_{L^2(\omega)}.
    $$
    Since we have chosen $\psi_{\lambda_k}$ such that the orthogonality condition \eqref{eq:ortho} is satisfied, we have
    $$
    (B^* \Phi_{\lambda_k}, B^* \Psi_{\lambda_k})_{L^2(\omega)}=0.
    $$
    Finally, we found a biorthogonal family to $\{B^* \Phi_{\lambda_k}, B^* \Psi_{\lambda_k}\}$ in $L^2(\omega)$.
    Thus, it only remains to solve \eqref{eq:mom2simpl}, that is
    $$
    \left \{
    \begin{aligned}
        \gamma_{\lambda_k} &= -e^{-\lambda_k T}(y^0, \Phi_{\lambda_k})_{L^2(0,\pi)} \\
        \theta_{\lambda_k} &= -e^{-\lambda_k T}(y^0, \Psi_{\lambda_k})_{L^2(0,\pi)}.
    \end{aligned}
    \right.
    $$
\end{itemize}
Now, we want to prove that, when the initial data $y^0$ satisfies
\begin{equation} \label{eq:y0bound}
\forall \lambda_k \in \Lambda, \: \left \{
    \begin{aligned}
        \vert (y^0, \Phi_{\lambda_k}) \vert &\leq C e^{-\lambda_k T_0} \\
        \vert (y^0, \Psi_{\lambda_k}) \vert &\leq C e^{-\lambda_k T_0}
    \end{aligned}
    \right.
\end{equation}
for a constant $C>0$, the series defined in \eqref{eq:control} with the parameters $\alpha_{\lambda_k}, \beta_{\lambda_k}, f_{\lambda_k}$ and $g_{\lambda_k}$ as above converges for all $T>0$.
We will need the following lemma, proved in \cite{boyermorancey}.

\begin{lemma} \label{lemmabound}
    \begin{enumerate}
        \item There exist $K \in \N$ and $C>0$ such that for all $k \geq K$, for all $F \in L^2(0,\pi,\R)$ and for all $u$ solution to 
        $$
        (\mathcal{L} - \lambda_k)u = F
        $$
        we have
        \begin{equation}
            \mathcal{N}_k(F,\omega) \leq C(\sqrt{\lambda_k} \Vert u \Vert_{\omega} + \sqrt{\lambda_k}(\vert u(0) \vert + \vert u(\pi) \vert ) + \Vert F \Vert_{\omega}).
        \end{equation}
        \item There exist $K \in \N$ and $C>0$ such that for all $k \geq K$, for all $F \in L^2(0,\pi,\R)$ such that $\displaystyle \int_0^{\pi} F \phi_{\lambda_k} = 0$ there exists $u$ solution to
        \begin{equation}
            \left \{
            \begin{aligned}
                (\mathcal{L} - \lambda_k)u &= F \\
                u(0)=u(\pi) &=0
            \end{aligned}
            \right.
        \end{equation}
        such that 
        \begin{equation}
            (\sqrt{\lambda_k} \Vert u \Vert_{\omega} - \Vert F \Vert_{\omega}) \leq C \mathcal{N}_k(F,\omega).
        \end{equation}
    \end{enumerate}
\end{lemma}

We study term by term the convergence of the associated series over $\lambda \in \Lambda_1$.
First, we look at the term with $\alpha_{\lambda_k}$, using the estimate given by Theorem \ref{th:biortho}:
$$
\begin{aligned}
    \Vert \alpha_{\lambda_k} r_{\lambda_k, 0} B^* \Phi_{\lambda_k} \Vert_{L^2((0,T) \times \omega)} &\leq e^{-\lambda_k T} C e^{-\lambda_k T_0} \Vert r_{\lambda_k, 0} \Vert_{L^2(0,T)} \frac{\Vert B^* \Phi_{\lambda_k} \Vert_{L^2(\omega)}}{\Vert B^* \Phi_{\lambda_k} \Vert^2_{L^2(\omega)}} \\
    &\leq C e^{-\lambda_k (T+T_0)} \frac{K e^{T \frac{\lambda_k}{2} + K \sqrt{\lambda_k} + \frac{K}{T}}}{\Vert B^* \Phi_{\lambda_k} \Vert_{L^2(\omega)}} \\
\end{aligned}
$$
To pursue, we need to have a lower bound on $\Vert B^* \Phi_{\lambda_k} \Vert_{L^2(\omega)}$. The following result, Theorem IV.1.3 in \cite{coursboyer}, gives exactly what we need.
\begin{theorem} \label{th:boundphi}
    There exists $C_1(\omega)>0$ such that
    \begin{equation}
        \Vert \phi_{\lambda_k} \Vert_{L^2(\omega)}^2 > C_1(\omega), \; \forall \lambda_k \in \Lambda.
    \end{equation}
\end{theorem}

Thus, we obtain
\begin{equation}
        \Vert \alpha_{\lambda_k} r_{\lambda_k, 0} B^* \Phi_{\lambda_k} \Vert_{L^2((0,T) \times \omega)} \leq C e^{-\lambda_k (\frac{T}{2}+T_0)} \frac{e^{K \sqrt{\lambda_k}}}{\sqrt{C_1(\omega)}}.
\end{equation}
Applying the Young inequality, we finally get
\begin{equation}
    \Vert \alpha_{\lambda_k} r_{\lambda_k, 0} B^* \Phi_{\lambda_k} \Vert_{L^2((0,T) \times \omega)} \leq C_T e^{-T_0 \lambda_k}.
\end{equation}
The series $\displaystyle \sum_{\lambda_k \in \Lambda_1} e^{-\lambda_k T_0}$ is convergent for any $T>0$, so by comparison, we can conclude that our series
$$
\sum_{\lambda_k \in \Lambda_1} \alpha_{\lambda_k} r_{\lambda_k, 0}(\cdot) f_{\lambda_k}(\cdot)
$$
converges in $L^2(0,T)$ for any $T>0$.
\vspace{1em}

Now, we do the same analysis on the term with $\beta_{\lambda_k}$ and get: 
\begin{multline} \label{eq:betabound}
    \Vert \beta_{\lambda_k} r_{\lambda_k, 1}g_{\lambda_k} \Vert_{L^2((0,T) \times \omega)} \leq C e^{-\lambda_k (T+T_0)}\frac{e^{T\frac{\lambda_k}{2} + K\sqrt{\lambda_k}+\frac{K}{T}}}{\vert I_{\lambda_k}(p,q) \vert} \\
    \times \left(\frac{\Vert \Psi_{\lambda_k} \Vert_{L^2(0,\pi)^2} + \vert I_{\lambda_k}(p,q)\vert T}{\Vert B^* \Phi_{\lambda_k} \Vert_{L^2(\omega)}} + \frac{\Vert B^* \Psi_{\lambda_k} \Vert_{L^2(\omega)}}{\Vert B^* \Phi_{\lambda_k} \Vert_{L^2(\omega)}^3} \right)
\end{multline}

We need to find upper bounds for $\Vert \Psi_{\lambda_k} \Vert_{L^2(0,\pi)^2}$ and $\Vert B^* \Psi_{\lambda_k} \Vert_{L^2(\omega)} = \Vert \psi_{\lambda_k} \Vert_{L^2(\omega)}$. For this, we will simply find an upper bound for $\Vert \psi_{\lambda_k} \Vert_{L^2(0,\pi)}$. Indeed, $\Psi_{\lambda_k} = \begin{pmatrix}
    \psi_{\lambda_k} \\
    \phi_{\lambda_k}
\end{pmatrix}$ and we already know that $\Vert \phi_{\lambda_k} \Vert_{L^2(0,\pi)} \leq 1$. \\
We will need the following result from \cite{coursboyer}.

\begin{lemma} \label{th:upboundpsi}
    Let $f: (0,\pi) \mapsto \R$ be a continuous function and $\lambda_k \geq 1$. Suppose that $u: (0,\pi) \mapsto \R$ satisfies the following equation
    $$
    \mathfrak{L} u = \lambda_k u + f.
    $$
    We define, for $x\in (0,\pi)$, the vectors
    $$
    U(x)\coloneqq \begin{pmatrix}
        u(x) \\
        \sqrt{\frac{\gamma(x)}{\lambda_k}} u'(x)
    \end{pmatrix} \textnormal{ and } F(x) = \begin{pmatrix}
        0 \\
        - \frac{f(x)}{\sqrt{\gamma(x) \lambda_k}}.
    \end{pmatrix}
    $$
    Then, there exists $C>0$ independent of $\lambda_k$ such that for any $x,y \in (0,\pi)$ we have 
    \begin{equation} \label{upboundpsi}
        \Vert U(x) \Vert \leq C \left(\Vert U(y) \Vert + \left\vert \displaystyle \int_x^y \Vert F(s) \Vert ds \right \vert \right).
    \end{equation}
\end{lemma}
Let us take $\widetilde{\psi}_{\lambda_k}$ the solution of 
\begin{equation} \label{eq:psi}
    (\mathfrak{L} - \lambda_k)\widetilde{\psi}_{\lambda_k} = I_{\lambda_k}(p,q)\phi_{\lambda_k} + \partial_x (p\phi_{\lambda_k}) - q\phi_{\lambda_k}
\end{equation}

such that $\widetilde{\psi}_{\lambda_k}(0) = \widetilde{\psi}_{\lambda_k}'(0)=0$. \\
Applying Lemma \ref{th:upboundpsi} with $y=0$, $u= \widetilde{\psi}_{\lambda_k}$ and $f = I_{\lambda_k}(p,q)\phi_{\lambda_k} + \partial_x(p\phi_{\lambda_k}) - q\phi_{\lambda_k}$, we get for all $x\in (0,\pi)$
$$
\Vert U(x) \Vert \leq C\left(\underbrace{ \Vert U(0) \Vert}_{=0} + \displaystyle \int_0^x \Vert (I_{\lambda_k}(p,q)\phi_{\lambda_k} + \partial_x(p\phi_{\lambda_k}) - q\phi_{\lambda_k})(s) \Vert ds \right).
$$
That is
$$
\vert \psi_{\lambda_k} (x) \vert + \sqrt{\frac{\gamma(x)}{\lambda_k}} \vert \psi_{\lambda_k}'(x) \vert \leq C \Vert I_{\lambda_k}(p,q)\phi_{\lambda_k} + \partial_x(p\phi_{\lambda_k}) - q\phi_{\lambda_k} \Vert_{L^2(0,\pi)}.
$$
We know that 
\begin{itemize}
    \item $\Vert \partial_x \phi_{\lambda_k} \Vert_{L^2(0,\pi)} \leq \sqrt{\lambda_k}$,
    \item $\Vert \phi_{\lambda_k} \Vert_{L^2(0,\pi)} \leq 1$,
    \item $\vert I_{\lambda_k}(p,q) \vert \leq \Vert q \Vert_{L^\infty(0,\pi)} + \frac{\Vert \partial_x p \Vert_{L^\infty(0,\pi)}}{2}$,
    \item $\Vert q \Vert_{L^\infty(0,\pi)}$, $\Vert p \Vert_{L^\infty(0,\pi)}$ and $\Vert \partial_x p \Vert_{L^\infty(0,\pi)}$ are bounded.
\end{itemize}
Hence, we get the following estimate, with $C>0$,
\begin{equation} \label{eq:boundpsipi}
    \Vert \psi_{\lambda_k} \Vert_{L^2(0,\pi)} \leq C\left(1+\sqrt{\lambda_k}\right).
\end{equation}
Since $\Vert \psi_{\lambda_k} \Vert_{L^2(\omega)} \leq \Vert \psi_{\lambda_k} \Vert_{L^2(0,\pi)}$, we also have
\begin{equation} \label{eq:boundpsiomega}
    \Vert \psi_{\lambda_k} \Vert_{L^2(\omega)} \leq C\left(1+\sqrt{\lambda_k}\right)
\end{equation}
Thus, \eqref{eq:betabound} gives
$$
\begin{aligned}
    \Vert \beta_{\lambda_k} r_{\lambda_k, 1}g_{\lambda_k} \Vert_{L^2((0,T) \times \omega)} &\leq C \frac{e^{- \lambda_k (\frac{T}{2} + T_0) + K\sqrt{\lambda_k}}}{\vert I_{\lambda_k}(p,q) \vert} \left(1+\sqrt{\lambda_k}\right) \\
    &\leq C_T \frac{e^{-\lambda_k (T_0 + \frac{T}{4})}}{\vert I_{\lambda_k} (p,q)\vert}  \left(1+\sqrt{\lambda_k}\right).
\end{aligned}
$$
By definition of $T_0$, for all $\varepsilon>0$ there exists a constants $C_{\varepsilon}>0$ such that for $k$ large enough
\begin{equation}
    e^{-\lambda_k (T_0 + \varepsilon)} \leq C_{\varepsilon} \vert I_{\lambda_k}(p,q) \vert.
\end{equation}
Then, we get, for $k$ large enough
\begin{equation}
    \begin{aligned}
        \Vert \beta_{\lambda_k} r_{\lambda_k, 1}g_{\lambda_k} \Vert_{L^2((0,T) \times \omega)} &\leq C_{T,\varepsilon} \frac{e^{-\lambda_k (T_0 + \frac{T}{4})}}{e^{-(T_0 + \varepsilon)}} (1+\sqrt{\lambda_k}) &\leq C_{T,\varepsilon} e^{-\lambda_k (\frac{T}{4} - \varepsilon)} (1 + \sqrt{\lambda_k}).
    \end{aligned}
\end{equation}
Taking $\varepsilon = \frac{T}{8}$ for example, we obtain that the series converges for all $T>0$. \\
It remains to study the convergence of the series on $\Lambda_2$. As in both previous cases, we find an estimate on the general term using Theorem \ref{th:biortho} and Theorem \ref{th:boundphi} on $\phi_{\lambda_k}$:
$$
\begin{aligned}
    \Vert r_{\lambda_k,0}[\gamma_{\lambda_k} h_{\lambda_k} + \theta_{\lambda_k} \Tilde{h}_{\lambda_k}] \Vert_{L^2((0,T) \times \omega)} &\leq K e^{-T \frac{\lambda_k}{2} + K \sqrt{\lambda_k} + \frac{K}{T}} \bigg[ \frac{\vert (y^0,\Phi_{\lambda_k}) \vert}{\Vert \phi_{\lambda_k} \Vert_{L^2(\omega)}^2} + \frac{\vert (y^0,\Psi_{\lambda_k}) \vert}{\Vert \psi_{\lambda_k} \Vert_{L^2(\omega)}}\bigg] \\
    &\leq  C e^{-\lambda_k (\frac{T}{2} + T_0) + K \sqrt{\lambda_k}} \bigg[ \frac{1}{C_1(\omega)} + \frac{\Vert \Psi_{\lambda_k} \Vert_{L^2(0,\pi)^2}}{\Vert \psi_{\lambda_k} \Vert_{L^2(\omega)}}\bigg].
\end{aligned}
$$
We need an upper bound for $\Vert \Psi_{\lambda_k} \Vert_{L^2(0,\pi)^2}$ and a lower bound on $\Vert \psi_{\lambda_k} \Vert_{L^2(\omega)}$. \\
First, we apply the same reasoning to find an upper bound for $\Vert \Psi_{\lambda_k} \Vert_{L^2(0,\pi)^2}$ as the one we used previously for $\Vert \Psi_{\lambda_k} \Vert_{L^2(0,\pi)^2}$, and we get the same estimate as in \eqref{eq:boundpsipi} so: 
\begin{equation}
    \Vert \Psi_{\lambda_k} \Vert_{L^2(0,\pi)^2} \leq C \left(1+\sqrt{\lambda_k}\right).
\end{equation}
For $\Vert \psi_{\lambda_k} \Vert_{L^2(\omega)}$, we apply Lemma \ref{th:upboundpsi} and get
\begin{dmath}
    \mathcal{N}_k(I_{\lambda_k}(p,q) \phi_{\lambda_k} -q\phi_{\lambda_k} +\partial_x(p\phi_{\lambda_k}),\omega) \leq C \left(\sqrt{\lambda_k} \Vert \psi_{\lambda_k} \Vert_{L^2(\omega)} + \Vert I_{\lambda_k}(p,q) \phi_{\lambda_k}-q\phi_{\lambda_k} + \partial_x(p\phi_{\lambda_k}) \Vert_{L^2(\omega)} \right).
\end{dmath}
Since we assumed $(\supp p \cup \supp q)\cap \omega = \emptyset$ we get
\begin{equation}
    \mathcal{N}_k(I_{\lambda_k}(p,q) \phi_{\lambda_k} -q\phi_{\lambda_k} +\partial_x(p\phi_{\lambda_k}),\omega) \leq C \left(\sqrt{\lambda_k} \Vert \psi_{\lambda_k} \Vert_{L^2(\omega)} + \vert I_{\lambda_k} (p,q) \vert \right).
\end{equation}
Hence, we obtain the following bound 
\begin{equation*}
    \begin{aligned}
        \Vert r_{\lambda_k,0}[\gamma_{\lambda_k} h_{\lambda_k} + \theta_{\lambda_k} \Tilde{h}_{\lambda_k}] \Vert_{L^2((0,T) \times \omega)} &\leq K e^{- \lambda_k (\frac{T}{2} + T_0) + K \sqrt{\lambda_k}} \bigg[ \frac{1}{C_1(\omega)} \\
        &+ \frac{C \left(1+ \sqrt{\lambda_k}\right)\sqrt{\lambda_k}}{\mathcal{N}_k(I_{\lambda_k}(p,q) \phi_{\lambda_k} -q\phi_{\lambda_k} +\partial_x(p\phi_{\lambda_k}),\omega)}\bigg] \\
        &\leq C_T \frac{\lambda_k e^{-\lambda_k (\frac{T}{2} + T_0) }}{\mathcal{N}_k(I_{\lambda_k}(p,q) \phi_{\lambda_k} -q\phi_{\lambda_k} +\partial_x(p\phi_{\lambda_k}),\omega)}
    \end{aligned}
\end{equation*}
Using the definition of $T_0$ and the Young inequality, as on the previous term, we get, for all $\varepsilon>0$ and $k$ large enough
\begin{equation}
    \Vert r_{\lambda_k,0}[\gamma_{\lambda_k} h_{\lambda_k} + \theta_{\lambda_k} \Tilde{h}_{\lambda_k}] \Vert_{L^2((0,T) \times \omega)} \leq C_{T,\varepsilon} \lambda_k e^{-\lambda_k ( \frac{T}{4} - \varepsilon)}.
\end{equation}
Again, taking $\varepsilon=\frac{T}{8}$ in sufficient to get that the series
$$
\sum_{\lambda_k \in \Lambda_2}  r_{\lambda_k,0}(\cdot)[\gamma_{\lambda_k} h_{\lambda_k}(\cdot) + \theta_{\lambda_k} \Tilde{h}_{\lambda_k}(\cdot)]
$$
converges in $L^2(0,T)$ for all $T>0$. \\
Then we have proved that the system with an initial data that satisfies \eqref{eq:y0bound} is null-controllable at any time $T>0$. Now, take any initial data $y^0$, and $T>T_0$. We can take a control $v$ such that $v(t,\cdot)=0$ for all $t \in (0, T_0)$. Then, the problem reduces to considering the same system but with initial condition at $T_0$, that satisfies \eqref{eq:y0bound}. Indeed, we have:
\begin{equation}
\forall t>0, \: \forall \lambda_k \in \Lambda_1, \: \left \{
    \begin{aligned}
        (y(t),\Phi_{\lambda_k}) &= e^{-t\lambda}(y^0, \Phi_{\lambda_k}) \\
        (y(t),\Psi_{\lambda_k}) &= e^{-t\lambda}(y^0, \Psi_{\lambda_k} - t I_{\lambda_k}(p,q) \Phi_{\lambda_k})
    \end{aligned}
    \right. \\
\end{equation}
and
\begin{equation}
\forall t>0, \: \forall \lambda_k \in \Lambda_2, \: \left \{
    \begin{aligned}
        (y(t),\Phi_{\lambda_k}) &= e^{-t\lambda}(y^0, \Phi_{\lambda_k}) \\
        (y(t),\Psi_{\lambda_k}) &= e^{-t\lambda}(y^0, \Psi_{\lambda_k}).
    \end{aligned}
    \right.
\end{equation}
Since $(y^0, \Phi_{\lambda_k})$ and $(y^0, \Psi_{\lambda})$ are bounded, using these equations with $t=T_0$ gives that
\begin{equation} 
\forall \lambda_k \in \Lambda, \: \left \{
    \begin{aligned}
        \vert (y(T_0), \Phi_{\lambda_k}) \vert &\leq C e^{-\lambda_k T_0} \\
        \vert (y(T_0), \Psi_{\lambda_k}) \vert &\leq C e^{-\lambda_k T_0}
    \end{aligned}
    \right.
\end{equation}
Applying what we did previously, we can conclude that the system is null controllable for all time $T>T_0$. 

    \item Now, it only remains to prove that this time $T_0$ is minimal for the null controllability. \\
Let $T>0$. Assume that the system is null controllable at time $T$. Then, for any $y^0$, we have a control $v$, and from \eqref{eq:momphi} and \eqref{eq:mompsi2}, we get $\forall \lambda \in \Lambda_1$
$$
\begin{aligned}
    \vert (y^0, e^{-T\lambda_k} \Psi_{\lambda_k})_{L^2(0,\pi)} \vert &=  \displaystyle \int_0^T \left( v(t,\cdot), \psi_{\lambda_k} + t I_{\lambda_k} (p,q) \phi_{\lambda_k} \right)_{L^2(\omega)} e^{-\lambda_k (T-t)} dt \\
    & \leq \Vert v \Vert \sqrt{T} (\Vert \psi_{\lambda_k} \Vert_{L^2(\omega)} + T \vert I_{\lambda_k} (p,q) \vert \Vert \phi_{\lambda_k} \Vert_{L^2(\omega)}).
\end{aligned}
$$
Since the system is null controllable from any initial data, we have a bound on the control function $v$; there exists $C>0$ such that 
\begin{equation} \label{eq:boundv}
    \Vert v \Vert \leq C \Vert y^0 \Vert.
\end{equation}
We use that and get
$$
    \vert (y^0, e^{-T\lambda_k} \Psi_{\lambda_k})_{L^2(0,\pi)} \vert \leq C_T \Vert y^0 \Vert (\Vert \psi_{\lambda_k} \Vert_{L^2(\omega)} + T \vert I_{\lambda_k} (p,q) \vert \Vert \phi_{\lambda_k} \Vert_{L^2(\omega)}). 
$$
Now, take $y^0 = \Psi_{\lambda_k}$, and remark that $\Vert \Psi_{\lambda_k} \Vert_{L^2(0,\pi)} \geq 1$. Then
\begin{equation} \label{eq:majorationnc}
    e^{-\lambda_k T} \leq C_T (\Vert \psi_{\lambda_k} \Vert_{L^2(\omega)} + T \vert I_{\lambda_k} (p,q) \vert \Vert \phi_{\lambda_k} \Vert_{L^2(\omega)}).
\end{equation}
Moreover, using Lemma \ref{lemmabound} and the fact that we assumed $(\supp p \cup \supp q) \cap = \emptyset$, we get
$$
\Vert \psi_{\lambda_k} \Vert_{L^2(\omega)} + T \vert I_{\lambda_k} (p,q) \vert \Vert \phi_{\lambda_k} \Vert_{L^2(\omega)} \leq C \mathcal{N}_{k}(I_{\lambda_k}(p,q) \phi_{\lambda_k} -q\phi_{\lambda_k} +\partial_x(p\phi_{\lambda_k}), \omega).
$$
Replacing in \eqref{eq:majorationnc}, we finally get
\begin{equation} \label{eq:ccl1}
    T \geq - \frac{\ln C_T + \ln \mathcal{N}_k (I_{\lambda_k}(p,q) \phi_{\lambda_k} -q\phi_{\lambda_k} +\partial_x(p\phi_{\lambda_k}), \omega)}{\lambda_k}.
\end{equation}
Now, take $\lambda_k \in \Lambda_2$. We make the same kind of computations as before and get
$$
\vert (y^0, \Psi_{\lambda_k}) e^{-\lambda_k T} \vert \leq \Vert v \Vert \sqrt{T} \Vert \psi_{\lambda_k} \Vert_{L^2(\omega)}.
$$
Moreover, $\mathcal{N}_k(I_{\lambda_k}(p,q) \phi_{\lambda_k} -q\phi_{\lambda_k} +\partial_x(p\phi_{\lambda_k}), \omega) \neq 0$ so we can write
$$
\frac{\vert (y^0, \Psi_{\lambda_k}) e^{-\lambda_k T} \vert}{\sqrt{\lambda_k} \mathcal{N}_k(I_{\lambda_k}(p,q) \phi_{\lambda_k} -q\phi_{\lambda_k} +\partial_x(p\phi_{\lambda_k}), \omega)} \leq \frac{C \Vert y^0 \Vert \sqrt{T} \Vert \psi_{\lambda_k} \Vert_{L^2(\omega)}}{\sqrt{\lambda_k} \mathcal{N}_k(I_{\lambda_k}(p,q) \phi_{\lambda_k} -q\phi_{\lambda_k} +\partial_x(p\phi_{\lambda_k}), \omega)}.
$$
Taking $y^0 = \Psi_{\lambda_k}$, and using Lemma \ref{lemmabound} we get that
$$
    e^{-\lambda T} \leq C \sqrt{\lambda_k} \mathcal{N}_k(I_{\lambda_k}(p,q) \phi_{\lambda_k} -q\phi_{\lambda_k} +\partial_x(p\phi_{\lambda_k}), \omega)
$$
We deduce
\begin{equation} \label{eq:ccl2}
    T \geq - \frac{\ln{C \sqrt{\lambda_k}} + \ln \mathcal{N}_k (I_{\lambda_k}(p,q) \phi_{\lambda_k} -q\phi_{\lambda_k} +\partial_x(p\phi_{\lambda_k}), \omega)}{\lambda_k}.
\end{equation}

Now, by definition of $T_0$, we know that there exists an extraction $f: \N \rightarrow \N$ such that 
$$
\lim_{k \rightarrow \infty}{- \frac{\ln{\mathcal{N}_{f(k)}(I_{\lambda_k}(p,q) \phi_{\lambda_f(k)}+ \partial_x(p\phi_{\lambda_{f(k)}}) - q\phi_{\lambda_{f(k)}}, \omega)}}{\lambda_{f(k)}}} = T_0.
$$
There are two possibilities:
\begin{itemize}
    \item Whether there is an infinity of $\{f(k)\}$ in $\Lambda_1$, and in that case we use \eqref{eq:ccl1} to conclude that $T \geq T_0$.
    \item Or there is an infinity of $\{f(k)\}$ in $\Lambda_2$, and in that case we use \eqref{eq:ccl2} to conclude that $T \geq T_0$.
\end{itemize}
In both cases, we proved that if the system is null controllable at time $T$, then $T \geq T_0$. We can then conclude that the minimal null control time is exactly $T_0$.
\end{itemize}
\end{proof}
\clearpage
\bibliography{refs}

\begin{thebibliography}{1}

\bibitem{coursboyer}
Franck Boyer.
\newblock Controllability of linear parabolic equations and systems, 2016.

\bibitem{boyermorancey}
Franck Boyer and Morgan Morancey.
\newblock Distributed null-controllability of some 1d cascade parabolic
  systems, 2023.

\bibitem{boyer}
Franck Boyer and Guillaume Olive.
\newblock Approximate controllability conditions for some linear 1d parabolic
  systems with space-dependent coefficients.
\newblock {\em Mathematical control and related fields}, 4(3):263--287, 2014.

\bibitem{duprez}
Michel Duprez.
\newblock {Controllability of a 2 $\times$ 2 parabolic system with
  space-dependent coupling terms of order one.}
\newblock {\em ESAIM: COCV}, 23(4):1473--1498, 2017.

\end{thebibliography}
\end{document}